\documentclass{amsart}

\usepackage{amsfonts, amsmath, amsthm, amssymb, graphicx, enumerate, tensor}
\usepackage[all]{xy}
\usepackage{cite}

\newtheorem{theorem}{Theorem}[section]
\newtheorem{lemma}[theorem]{Lemma}
\newtheorem{cor}[theorem]{Corollary}
\newtheorem{prop}[theorem]{Proposition}

\newtheorem{definitiontemp}[theorem]{Definition}
\newenvironment{definition}{\begin{definitiontemp}
\normalfont}{\end{definitiontemp}}

\numberwithin{equation}{section}

\theoremstyle{remark}
\newtheorem{remark}[theorem]{Remark}

\newcommand{\act}{\ltimes}

\newcommand{\C}{\mathbb{C}}
\newcommand{\Z}{\mathbb{Z}}

\newcommand{\A}{\mathcal{A}}
\newcommand{\rep}{\text{Rep}}
\newcommand{\vect}{\text{Vec}}

\newcommand{\Oo}{\mathcal{O}}

\newcommand{\Cc}{\mathcal{C}}
\newcommand{\Mm}{\mathcal{M}}
\newcommand{\Dd}{\mathcal{D}}
\newcommand{\Nn}{\mathcal{N}}

\newcommand{\Bb}{\mathcal{B}}
\newcommand{\Hom}{\text{Hom}}

\newcommand{\Mod}{\text{Mod}}

\newcommand{\Pic}{\mathcal{P}\text{ic}}
\newcommand{\Bm}{\text{\underline{BrPic}}}
\newcommand{\Bmm}{\text{\underline{\underline{BrPic}}}}
\newcommand{\Pc}{\text{\underline{\underline{Pic}}}}
\newcommand{\Att}{\text{\underline{Aut}}^{\otimes}}
\newcommand{\At}{\text{\underline{Aut}}}
\newcommand{\Ot}{\text{Out}}
\newcommand{\Ott}{\text{\underline{\underline{Out}}}}
\newcommand{\RP}{\text{REP}^\times(K)}
\newcommand{\Rp}{\text{Rep}^\times(K)}
\newcommand{\drr}{\underline{\underline{\text{Der}}}}

\newcommand{\ok}{\vect_K}
\newcommand{\Zz}{\mathcal{Z}}
\newcommand{\Bimod}{\text{Bimod}}

\title{Extensions and duality.}
\author{Ilya Shapiro}

\begin{document}
\maketitle

\begin{abstract}
For a fixed finite group $Q$ and semi-simple finite dimensional algebra $S$, we examine an equivalence between strongly $Q$-graded algebras (extensions) with identity component $S$ and $S^1$-gerbes on action groupoids of $Q$ on the set of isomorphism classes of simple objects of the category of $S$-modules. This clarifies the nature of the map considered in \cite{gerbes}.  Motivated by this and \cite{eno2} we suggest and study a notion of extensions suitable to the case when $S$ is replaced by a Hopf algebra, in the sense that there is a bijection between extensions with ``fiber" $H$ and $H^*$.  In particular we focus on the case of $H$ equal to the group algebra of a finite group.  When $K$ is abelian, the answer is particularly symmetric as duality of Hopf algebras does not take us outside of the category of groups.
\end{abstract}
\bigskip
\noindent \emph{Keywords}: Extensions, fusion categories, higher groups, duality.

\section{Introduction.}
The duality of the title is actually two very different dualities.  The first refers to the bijection between algebra extensions and $S^1$-gerbes, while the second deals with ``extensions" of Hopf algebras and their relative duality.  There is a third kind of duality (more of a symmetry) that arises out of considerations of certain group actions on pointed fusion categories, it is the very opposite of relative.  The paper is organized as follows.

In Section \ref{algebras} we generalize and refine the approach of \cite{gerbes} into an actual bijection; although our setting here is somewhat less general, it can easily be ``re-generalized". More precisely, in \cite{gerbes} a construction of an $S^1$-gerbe on a (dual) Lie groupoid is constructed from a $K$-gerbe on the original Lie groupoid.  It is then shown that the two have equivalent representation theories.  In this text, we consider instead of the original Lie groupoid, only the case of the $Q\act pt$, the groupoid quotient of a point by a finite group $Q$.  The nature of the correspondence (in \cite{gerbes}) between $K$-gerbes on the original and $S^1$-gebres on the dual is not examined.  In fact it is not hard to see that it is not a bijection.  The mystery of the correspondence or the lack thereof is solved by replacing the $K$-gerbes with something else; this is the content of Theorem \ref{bijection}.  The idea to consider $Q$-extensions of algebras, that is so obvious in retrospect, occurred to us while reading \cite{eno2}. The relationship between $Q$-extensions of an algebra ($\C K$ to be precise) and $K$-gerbes on  $Q\act pt$ (known more commonly as group $K$-extensions of $Q$) is explained in Section \ref{grp}. In Section \ref{fusioncats} we are motivated by an attempt to define a notion of $Q$-extension, of a semi-simple finite dimensional Hopf algebra $H$, that is stable under duality of Hopf algebras, in the relative sense.  What we actually obtain is a slightly different notion of a $Q$-extension of a pair $(\Cc,\Mm)$ of a fusion category and its module.  We further focus on the case of $H=\C K$ for $K$ a finite group, which in our framework corresponds to $(\rep(K),\vect)$. In Section \ref{symmetry} we look at an interesting symmetry between  $Q$ actions on $\ok$ and $K$ actions on $\vect_Q$.  In the Appendix, i.e., Section \ref{appx} we collect mostly known constructions and observations that are needed in the main body of the paper.  Our notation is not always standard and so we feel that the inclusion of the Appendix addresses this issue as well.

\subsection{Notation} Possible source of confusion: $H^i(K,A)$ denotes the $i$th cohomology group of $K$ with coefficients in an abelian group $A$ on which $K$ acts by group automorphisms. If the group $A$ is $\C^\times$ then we write simply $H^i(K)$.  On the other hand $H^i(G;K,A)$ is the relative to $K$ $i$th group cohomology of $G$ with coefficients in $A$, and again $H^i(G;K)$ means that $A=\C^\times$.

What we call an $S^1$-gerbe should be more appropriately called a $\C^\times$-gerbe or a $\C^\times$-central extension.  We owe our notation to \cite{gerbes}.

\section{The $Q$-extensions of algebras and duality}\label{algebras}
Let $S$ be a semi-simple finite dimensional algebra over $\C$.  Let $Q$ be a finite group (it can be a groupoid, nothing would change).  As usual, we say that $R$ is a $Q$-extension of $S$, if $$R=\bigoplus_{q\in Q}R_q$$ is strongly $Q$-graded and $R_1=S$.   This is equivalent to a homomorphism $$Q\rightarrow \Bm(S)$$ where the latter is the $2$-group of invertible $S$-bimodules.\footnote{We invite the reader to compare the discussion contained in this section to \cite{obstructions}.}  Note that  $\Bm(S)=\At(S\text{-mod})$. Let $I_S$ denote the set of isomorphism classes of simple $S$-modules,  and observe that $$\pi_1(\Bm(S))=Aut(I_S)$$ and if we denote by $\Oo_{I_S}$ the functions on the set $I_S$, $$\pi_2(\Bm(S))=Z(S)^\times=\Oo^\times_{I_S}.$$ Let us make the choices of $V_i\in i\in I_S$ arbitrarily, thus $$S\simeq \bigoplus_{i\in I_S}End(V_i).$$ Observe that for $\sigma\in Aut(I_S)$ we have an element $X_\sigma\in \Bm(S)$, more precisely $$X_\sigma=\bigoplus_{i\in I_S}V_{\sigma(i)}\otimes V^*_i.$$  Note that we have a canonical identification $$X_\sigma\otimes_S X_{\sigma'}\simeq X_{\sigma\sigma'}$$ and so the choices of $V_i$'s that we made split the $2$-group.  More precisely, we have:
\begin{lemma}
As $2$-groups, $$\Bm(S)\simeq T(\Oo^\times_{I_S})\rtimes Aut(I_S)$$ where $T(\Oo^\times_{I_S})$ denotes the $2$-group of $\Oo^\times_{I_S}$-torsors or equivalently $\Pic_{I_{S}}$, the $2$-group of line bundles on $I_S$ with the usual $\otimes$ product.
\end{lemma}

\begin{definition}
Given a homomorphism $Q\rightarrow\A$ from a finite group $Q$ to a $2$-group (or a $3$-group) $\A$, call the resulting homomorphism $Q\rightarrow \pi_1(\A)$ its band.  This is done by analogy with gerbes.
\end{definition}

Thus, given a $Q$-extension $R$ of $S$,  the composition $$Q\rightarrow \Bm(S)\rightarrow Aut(I_S)$$ is the band of $R$.  By the discussion above, the  $Q$-extensions of $S$ with the band $\rho: Q\rightarrow Aut(I_S)$ are parameterized by the twisted homomorphisms $$Q\rightarrow T(\Oo^\times_{I_S}),$$ i.e., the second cohomology group $H^2(Q,\Oo^\times_{I_S})$ where the action of $Q$ on $\Oo^\times_{I_S}$ is via its action on $I_S$ via $\rho$.

\begin{lemma}
We have an isomorphism of abelian groups $$H^2(Q,\Oo^\times_{I_S})\simeq H^2(Q\act I_S,\C^\times),$$ where $Q\act I_S$ is the action groupoid of $Q$ on $I_S$.
\end{lemma}
\begin{proof}
The cohomology isomorphism is a case of the general observation that for an action of $Q$ on a finite set $X$ we have an isomorphism of complexes $$C^n(Q,\Oo^\times_X)\simeq C^n(Q\act X,\C^\times)$$ given by \begin{center}
\includegraphics[height=.25in]{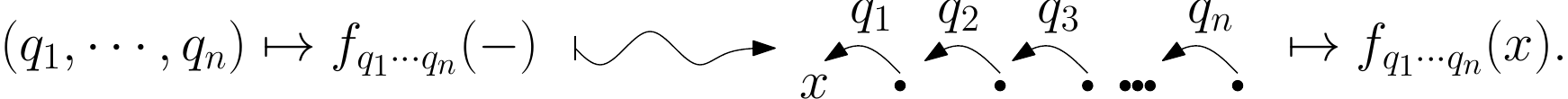}
\end{center}
\end{proof}

Note that $H^2(Q\act I_S,\C^\times)$ parameterizes $S^1$-gerbes on $Q\act I_S$,  and so we have almost proved the following.

\begin{theorem}\label{bijection}
For any action $\rho$ of $Q$ on the set $I_S$, we have a bijection between $Q$-extensions $R$ of $S$ with band $\rho$, and $S^1$-gerbes $\widehat{Q\act I_S}$ on the action groupoid of $\rho$.   Furthermore, as abelian categories  $$R\text{-mod}\simeq\rep'\widehat{Q\act I_S}$$ with the latter denoting $S^1$-twisted representations of the action groupoid of $\rho$.
\end{theorem}
Note that the equivalence above is subject to the choices of $V_i\in i\in I_S$.

\begin{remark}
Observe that $S^1$-gerbes on the action groupoid of $\rho$ can be explicitly described as follows. Let $i_1,\cdots, i_r$ be orbit representatives, let $Q_j=\text{Stab}_Q(i_j)$, then collections $\{\widehat{Q}_j\}$ of $\C^\times$-central extensions of $Q_j$ correspond to $S^1$-gerbes.
\end{remark}

\begin{proof}
Indeed, it is immediate that $$R\text{-mod}\simeq S\text{-mod}^Q$$ after recalling that the map $Q\rightarrow\At(S\text{-mod})$ is equivalent to the data of the extension $R$.  The rest is a consequence of Lemma \ref{invar}.
\end{proof}

\begin{lemma}\label{invar}
Let $\Cc$ be a finite dimensional $2$-vector space over $\C$, i.e., an abelian category non-canonically equivalent to a finite number of copies of $\vect$.  Assume that a finite group $Q$ acts on $\Cc$, more precisely we are given a map of $2$-groups $Q\rightarrow\At (\Cc)$.  Then $$\Cc^Q\simeq \rep'\widehat{Q\act I_{\Cc}}.$$
\end{lemma}

The construction of $\widehat{Q\act I_{\Cc}}$ proceeds as follows.  Let $I_{\Cc}$ denote the set of isomorphism classes of simple objects of $\Cc$, so that we have an action of $Q$ on the set $I_{\Cc}$.  Choose representatives $c_i\in i\in I_{\Cc}$, you may notice that we keep doing this. We define an $S^1$-gerbe on the action groupoid by decorating the arrow from $i$ to $j$ labeled by $q$ with $$\ell_{q,i}:=Hom_{\Cc}(c_j,q\cdot c_i).$$

\begin{proof}
Let us sketch a proof.  If $M\in \Cc^Q$, then $M\in \Cc$ and it comes equipped with suitably compatible isomorphisms $\varphi_q:q\cdot M\simeq M$.  Let $M_i=Hom_{\Cc}(c_i, M)$ so that $M\simeq\bigoplus_i M_i c_i$ and the isomorphisms $\varphi_q$ yield  $\varphi_{q,i}\in Hom_{\Cc}(M_i q\cdot c_i, M_j c_j)=M_i^*\otimes M_j\otimes\ell^*_{q,i}$, i.e., we obtain $\varphi_{q,i}:\ell_{q,i}\otimes M_i\simeq M_j$ and so an element of $\rep'\widehat{Q\act I_{\Cc}}$. For the converse, we reverse the procedure so that again $M\simeq\bigoplus_i M_i c_i$ and $\varphi_q=\bigoplus_i\varphi_{q,i}$.
\end{proof}

\begin{remark}
Using the Lemma \ref{invar}, the category $\Cc^Q$ can be described more compactly, albeit even less canonically.  Namely, let $X$ be the set of orbits of $Q$ acting on $I_{\Cc}$, and let $x$ denote both an element in $X$ and an arbitrary fixed choice of its representative in   $I_{\Cc}$.  Let $$\widehat{Q}_x=\{(q,c)| q\cdot x=x,\,\,c\in\text{Iso}(c_x,q\cdot c_x)\}$$ which is a $\C^\times$-central extension of the stabilizer in $Q$ of $x\in X$.  Then \begin{equation}\label{frules}\Cc^Q\simeq\bigoplus_{x\in X}\rep'(\widehat{Q}_x).\end{equation}  The formula \eqref{frules} was originally shown in \cite{fusionrules}.
\end{remark}

Note that the construction of the $\widehat{Q\act I_{\Cc}}$ above from the action map $Q\rightarrow\At (\Cc)$ provides an explicit version of the cohomologically obtained equivalence between extensions and $S^1$-gerbes.  Simply let $\Cc=S\text{-mod}$.  More generally, with $\Cc$ as above, we have (using the choice of $c_i\in I_{\Cc}$): $$\At(\Cc)\simeq \Pic_{I_{\Cc}}\rtimes Aut(I_\Cc).$$ Thus the homomorphisms $Q\rightarrow\At(\Cc)$ with band $\rho: Q\rightarrow Aut(I_\Cc)$ are in bijection with homomorphisms $Q\act I_{\Cc}\rightarrow \Pic$ where the latter $2$-group is equivalent to $\At(\vect)$.  We can therefore restate Theorem \ref{bijection} as follows:

\begin{prop}
For any action $\rho: Q\rightarrow Aut(I_\Cc)$ there is a bijection between actions of $Q$ on $\Cc$ with band $\rho$ and actions of $Q\act I_{\Cc}$ on $\vect$.  Furthermore, $$\Cc^Q\simeq\vect^{Q\act I_{\Cc}}.$$
\end{prop}

\subsection{Group extensions}\label{grp}
Let $K$ be a finite group, we may consider $S=\C K$ so that $S\text{-mod}=\rep(K)$. Group extensions $G$ of $Q$ by $K$ provide examples of $Q$-extensions of $\C K$ via $G\mapsto \C G$. Namely, we have a morphism of $2$-groups: \begin{align*}B(K)&\rightarrow\Bm(\C K)\\X&\mapsto\C X \end{align*} where $B(K)$ is the $2$-group of $K$-bitorsors. Note that group extensions are the same as maps $Q\rightarrow B(K)$, namely if $G$ is a group extension then $q\mapsto G_q\in B(K)$. However the functor from group extensions to the algebra extensions is not particularly well behaved.  This can be seen cohomologically, i.e., all extensions of $Q$ by $K$ with a fixed band $Q\rightarrow \Ot(K)$ form an $H^2(Q,Z(K))$ torsor, if there are any.  However $Q$-extensions of $\C K$, with the same band now viewed as $Q\rightarrow Aut (I_K)$, are parameterized by $H^2(Q,Z(\C K)^\times)$.  It is true nevertheless that the categories of modules are preserved, i.e., $\rep(G)=\rep(\C G)$.  Following up the functor from group extensions to ring extensions by the equivalence between the latter and $S^1$-gerbes, one obtains the map studied in \cite{gerbes}.  It also explains the failure of that map to be an equivalence in any sense.

\begin{remark}
To illustrate the difference between group extensions and algebra extensions, consider the example of the standard non-trivial extension $$\Z/p\rightarrow \Z/p^2\rightarrow\Z/p.$$  The corresponding algebra extension of $\C\Z/p$ by $\Z/p$ has to be trivial, since the band action of $\Z/p$ on $(\Z/p)^*$ is trivial, and the stabilizers, which are all $\Z/p$, have no non-trivial $\C^\times$-central extensions. Indeed, by the same argument we see that $\Z/m$-extensions of $S$ are parametrized by their bands alone. Indeed, since no subgroup of a cyclic group, itself being cyclic, can have any non-trivial $\C^\times$-central extensions, so the action groupoid $\Z/m\act I_S$ has no non-trivial $S^1$-gerbes on it.
\end{remark}

A twisted version is also considered in \cite{gerbes}, namely, $S=\C_{\phi}K$ where $\phi$ is a $2$-cocycle in the group cohomology of $K$. In this case $S\text{-mod}=\rep^{\phi}(K)$.  Examples of such extensions are provided by group extensions of the form \begin{equation}\label{twist}\widehat{K}\rightarrow\widehat{G}\rightarrow Q\end{equation} where $\widehat{K}$ is a $\C^\times$ central extension of $K$ given by $\phi$ so that $$\C_{\phi}K=\C\times_{\C^\times}\widehat{K}.$$  And $\widehat{G}$ is given by an extension of $\phi$ to all of $G$.  Again the map studied in \cite{gerbes} is obtained via the equivalence between extensions and gerbes.  More precisely, from \eqref{twist}, we obtain a map $Q\rightarrow B'(\widehat{K})$ with $q\mapsto\widehat{G}_q\in B'(\widehat{K})$, where the latter consists of $\widehat{K}$-bitorsors $X$ with $xc=cx$ for $c\in\C^\times\subset\widehat{K}$ and $x\in X$.  Observe that we have \begin{align*}B'(\widehat{K})&\rightarrow\Bm(\C_\phi K)\\X&\mapsto \C\times_{\C^\times} X\end{align*} so that a group extension \eqref{twist} produces an extension of $\C_\phi K$  by $Q$, namely $\C_\phi G$.

As the equivalence preserves categories of modules this yields a different point of view on the Clifford theory that describes representations of a group via some structures on the representation category of its normal subgroup.

\section{The quasi-trivial $Q$-extensions of dual fusion categories}\label{fusioncats}

By examining the case of $S=\C K$ above we are led to the consideration of extensions $G$ of $Q$ by $K$ from the point of view of the resulting action of $Q$ on $\rep(K)$ that also determines the extension.  As before we have that $$\rep(K)^Q\simeq\rep(G).$$  But $\C K$ is not just an algebra, but a Hopf algebra which endows $\rep(K)$ with a tensor structure.  If we were to consider an arbitrary $Q$-extension of $\C K$ then the result has no reason to also be a Hopf algebra just as if the action of $Q$ on $\rep(K)$ does not preserve the tensor then $\rep(K)^Q$ need not have one.

Thus a definition of a $Q$-extension of a Hopf algebra $H$  (semi-simple and finite dimensional) suggests itself.  Namely, it is the data of $$Q\rightarrow\Att(\rep(H)).$$  Note that the pair $(\rep(H),\vect)$ consisting of a fusion category and a fiber functor completely determines $H$ as a Hopf algebra.  Thus, since $\rep(H)^Q$ is now both a fusion category and has a fiber functor, it determines a Hopf algebra $T$ which could be called a $Q$-extension of $H$ and we would have $$\rep(T)=\rep(H)^Q$$ by definition, as fusion categories with fiber functors. The category $\rep(T)$ could be understood, as above, in terms of twisted representations of a certain $S^1$-gerbe on $Q\act I_H$.  However one should point  out that this perspective is poorly suited to seeing the tensor structure on the representations, though \cite{fusionrules} deals exactly with this problem.

There are certain issues with this definition if one wants particular things to be true.  The lesser problem is that for $H=\C K$, the $2$-group $\Att(\rep(K))$ is strictly larger, in general, than $B(K)$ the $2$-group of $K$-bitorsors.  Since it is the homomorphisms from $Q$ into the latter that yield group extensions of $Q$ by $K$, we would have to live with ``quantum" extensions as well.  This problem (and some may not consider it as such) can easily be fixed.  Denote by $\Att(\rep(H),\vect)$ the full subcategory of objects that preserve the canonical fiber functor.  Then since $$\Att(\rep(K),\vect)\simeq B(K)$$ we recover, in the group case, the usual notion of extension.  Observe that this would work for weak Hopf algebras as well, namely given a homomorphism $Q\rightarrow \Att(\rep(H), R\text{-mod})$, where $R$ is the base of the weak Hopf algebra $H$,  we can define the extension weak Hopf algebra $T$ as obtained \cite{modulecategories, eno1} from the data $(\rep(H)^Q, R\text{-mod}, R)$.

The bigger problem is that this definition is, in either of its forms, not self-dual ``relative to $Q$".  More precisely, the $Q$-extensions of $H$ are not in bijection with the $Q$-extensions of the dual Hopf algebra $H^*$, using either definition.  For example, examining the second version, we see that for $H=\C K$, we have $\Att(\rep(K),\vect)\simeq B(K)$ whereas $$\Att(\rep(\Oo_K),\vect)\simeq K^*[1]\rtimes \text{Aut}(K).$$  Note that $K^*[1]$ stands for the group of characters of $K$ shifted up by a degree, i.e., it is the $2$-group of $K^*$-torsors.

More concretely, if we have an extension $$A\rightarrow G\rightarrow Q$$ with $A$ an abelian group (to stay inside groups for duality), there is no dual of $G$ relative to $Q$ in groups, namely there is no natural $$A^*\rightarrow ``G^{*Q}"\rightarrow Q.$$  On the other hand, if we consider this problem in the setting of fusion categories then a solution presents itself.  Namely, the corresponding to $G$ object is the fusion category $\vect_G$ of $G$-graded vector spaces with convolution.  It has a module category $\vect_Q$ and the dual with respect to this module, see \cite{modulecategories}, is $S_A(G)$ the category of $A$-biequivariant sheaves on $G$ which is equivalent (see Section \ref{stuff1}) to $\vect^\omega_{A^*\rtimes Q}$ where $\omega\in H^3(A^*\rtimes Q, \C^\times)$ encodes the non-triviality of the original extension $G$. More succinctly, \begin{equation}\label{dual}\vect_G^{*\vect_Q}\simeq\vect^\omega_{A^*\rtimes Q}.\end{equation}Thus a non-trivial extension with a trivial cocycle is dual to a trivial extension with a non-trivial cocycle.

More generally, for a group extension $G$ of $Q$ by a possibly non-abelian $K$, we have that  $\vect_G$ is dual (via $\vect_Q$) to $S_K(G)\simeq \rep(K)\rtimes Q$.  Thus we are looking for a definition of extension that unifies the two examples: $\vect_G$ and $\rep(K)\rtimes Q$.

The key to the answer is contained in \cite{eno2}, where a notion of a $Q$-extension of a fusion category $\Cc$ is studied as a higher version of the map $Q\rightarrow \Bm (S)$ used in the previous section.  More precisely, for a fusion category $\Cc$ they consider the $3$-group $\Bmm(\Cc)$ of invertible $\Cc$-bimodule categories.  Thus an extension is a homomorphism of $3$-groups $Q\rightarrow \Bmm(\Cc)$.  It can also be defined as $\Dd=\bigoplus_{q\in Q}\Dd_q$ with $\Dd_1=\Cc$ and suitable additional requirements.  This notion of extension dualizes very well, i.e., if $\Mm$ is a $\Cc$-module category then \begin{equation}\label{brpic}\Bmm(\Cc)\simeq\Bmm(\Cc^{*\Mm})\end{equation} via the map $\Nn\mapsto \Mm^{op}\boxtimes_{\Cc}\Nn\boxtimes_{\Cc}\Mm$.  Equivalently, if $\Dd$ is a $Q$-extension of $\Cc$, then $\Dd\boxtimes_{\Cc}\Mm$ is a $\Dd$-module category and $\Dd^{*\Dd\boxtimes_{\Cc}\Mm}$ is the dual $Q$-extension of $\Cc^{*\Mm}$.  Compare this with \cite{weaklygroup}.

Unfortunately, $\Bmm(\Cc)$ is quite unwieldy (for our purposes) and $\Ott(\Cc)$ of \cite{eno2} suites us much better.  It is roughly $[\Cc\rightarrow \Att(\Cc)]$ in the same way that $B(K)$ is $[K\rightarrow \text{Aut}(K)]$.  More concretely, it is the subgroup of $\Bmm(\Cc)$ that consists of quasi-trivial invertible $\Cc$-bimodules, namely those that are equivalent to $\Cc$ itself as left $\Cc$-modules.  One checks that they are all of the form $\Cc^\sigma$ where the right action of $\Cc$ on itself has been modified by a $\sigma\in\text{Aut}^\otimes(\Cc)$, while the left action remained unchanged.  Thus there is a homomorphism $$\Att(\Cc)\rightarrow\Ott(\Cc)$$ that allows us to modify the previous attempt at a definition of a $Q$-extension of $\Cc$ (viewed as the category of representations of a Hopf algebra) from $Q\rightarrow \Att(\Cc)$ to $Q\rightarrow \Ott(\Cc)$.

We note that $\Ott(\Cc)$ is not preserved by the duality above. Indeed let $\omega\in Z^2(K,\C^\times)$ be a $2$-cocycle, then it gives a tensor autoequivalence of $\ok$, yet the dual of $\ok\,^\omega$ with respect to the canonical fiber functor is $\rep^\omega(K)\ncong \rep(K)$ as  left $\rep(K)$-modules.

Fortunately, the fix is one we have seen before, namely we consider $\Ott(\Cc,\Mm)$ instead.  This means that we only consider $\Nn\in\Ott(\Cc)$ with $\Nn\boxtimes_{\Cc}\Mm$ equivalent (via an unspecified equivalence) to the module $\Mm$.  To see that this is indeed a fix, observe that the duality sends $\Nn$ to $\Mm^{op}\boxtimes_{\Cc}\Nn\boxtimes_{\Cc}\Mm$ so that  $\Cc^{*\Mm}\simeq\Mm^{op}\boxtimes_{\Cc}\Mm$ is equivalent to the dual of $\Nn$, as a left module category, if and only if $\Nn\boxtimes_{\Cc}\Mm\simeq\Mm$ as left module categories. Furthermore, $\Ott(\Cc,\Mm)$  consists of all objects of $\Ott(\Cc)$ that dualize to objects of $\Ott(\Cc^{*\Mm})$.  Note that $$\Att(\Cc,\Mm)\rightarrow \Ott(\Cc,\Mm).$$

\begin{remark}\label{band}
To give an idea of the difference between $\Bmm(\Cc)$, $\Ott(\Cc)$ and $\Ott(\Cc,\Mm)$, let us consider the case of $\Cc=\vect_A$ for an abelian group $A$ with its canonical module $\vect$.  Then by \cite{eno2}, we have $\text{BrPic}(\vect_A)\simeq \Oo(A\oplus A^*,q)$, i.e., it is the orthogonal group of automorphisms of $A\oplus A^*$ with its natural quadratic form $q(a,\chi)=\chi(a)$.  On the other hand, $\text{Out}(\vect_A)$ is the subgroup of block lower triangular matrices in the orthogonal group: \begin{align*}H^2(A)\rtimes Aut(A)&\rightarrow\Oo(A\oplus A^*,q)\\(B,\phi)&\mapsto\begin{bmatrix}\phi&0\\\phi^{-1*}B& \phi^{-1*}\end{bmatrix}\end{align*} where $B$ is the skew-symmetric bicharacter indexing a cohomology class.  Lastly, $\text{Out}(\vect_A,\vect)\simeq Aut(A)$ consists of block diagonal matrices.

\end{remark}

Let us summarize the above in the following Proposition.

\begin{prop}\label{brpic}
Let $\Cc$ be a fusion category with $\Mm$ a module category, then ``conjugation" by $\Mm$ establishes a pair of equivalences of $3$-groups: $$\Bmm(\Cc)\simeq\Bmm(\Cc^{*\Mm})$$ and more importantly $$\Ott(\Cc,\Mm)\simeq\Ott(\Cc^{*\Mm},\Mm).$$
\end{prop}

Thus $\Ott(\Cc,\Mm)$ fixes the dualizing problem that $\Att(\Cc,\Mm)$ had, and so prompts:

\begin{definition}
Let $Q$ be a finite group and $\Cc$ a fusion category with a module category $\Mm$. Then an extension of $(\Cc,\Mm)$ by $Q$ is a graded fusion category $$\Dd=\bigoplus_q\Dd_q$$ with $\Dd_1\simeq\Cc$ \emph{and} given by a homomorphism $$Q\rightarrow \Ott(\Cc,\Mm).$$
\end{definition}

Note that $\Dd$ comes with a module category $\Nn:=\Dd\boxtimes_{\Cc}\Mm$.  Sometimes we refer to the pair $(\Dd,\Nn)$ as a $Q$-extension of $(\Cc,\Mm)$.

\begin{cor}
The map $$(\Dd,\Nn)\mapsto(\Dd^{*\Nn},\Nn)$$ establishes a bijection between $Q$-extensions of $(\Cc,\Mm)$ and $Q$-extensions of $(\Cc^{*\Mm},\Mm)$.
\end{cor}

Let us examine the case of weak Hopf algebras more carefully (as these are almost the same as pairs $(\Cc,\Mm)$). Let $H$ be a weak Hopf algebra with base $R$, then let $(\Cc,\Mm)=(\rep(H), R\text{-mod})$.  Applying the above machinery to this case we obtain, for a given $Q\rightarrow\Ott(\Cc,\Mm)$, a fusion category and module pair $(\Dd,\Nn)$.  The problem is that we can't canonically extract a weak Hopf algebra from it, as it requires a choice of a suitable element in $\Nn$ \cite{modulecategories, eno1}, and all we have is $R\in R\text{-mod}$.  While it is true that $R\text{-mod}$ is a full subcategory of $\Nn=\Dd\boxtimes_{\rep(H)}R\text{-mod}$, the choice of $R$ is no longer suitable.  The situation is no better for a Hopf algebra.  We are forced to compromise.

\begin{definition}
A $Q$-extension of a weak Hopf algebra $H$ is the data of a homomorphism $$Q\rightarrow \Ott(\rep(H), R\text{-mod}).$$
\end{definition}

Then we get:

\begin{cor}
Let $H$ be a semi-simple, finite dimensional weak Hopf algebra over $\C$.  If $Q$ is a finite group, then $Q$-extensions of $H$ are in bijection with $Q$-extensions of $H^*$.
\end{cor}

Note that if we do choose a suitable element in $\Nn$, then the two weak Hopf algebras that arise from $(\Dd,\Nn)$ and $(\Dd^{*\Nn},\Nn)$ are dual to each other.  However, even this non-canonical construction of an ``extension" weak Hopf algebra is unlike our previous attempt.  This can be seen from the two approaches to $G$, an extension of $Q$ by $K$.  The original method, that uses $Q\rightarrow \Att(\rep(\C K),\vect)$, produces the pair $(\rep(G),\vect)$ from which we obtain the Hopf algebra $\C G$.  Our new method produces the pair $(\rep(K)\rtimes Q, \vect_Q)$ which does have a natural choice of $\Oo_Q\in \vect_Q$, but the resulting weak Hopf algebra is nothing like $\C G$.  However, it is true that $\rep(G)$ and $\rep(K)\rtimes Q$ are dual fusion categories with respect to $\rep(K)$. (This is a special case of the duality between $\Cc\rtimes Q$ and $\Cc^Q$ via $\Cc$, see \cite{ni}.) It is thus natural to live with the fact that a $Q$-extension of a weak Hopf algebra is not a weak Hopf algebra but a pair $(\Dd,\Nn)$ with $\Dd_1=\rep(H)$ and $\Nn=\Dd\boxtimes_{\rep(H)}R\text{-mod}$.

To summarize, if one uses the $Q\rightarrow \Att(\Cc,\Mm)$ notion of extension then the extension itself should be $\Cc^Q$, whereas composing the above with $\Att(\Cc,\Mm)\rightarrow \Ott(\Cc,\Mm)$ and so using the last definition one has $\Cc\rtimes Q$ as the extension.  The former is good for giving us precisely the extensions we want, but the latter has very good duality properties.

\begin{remark}
One may try to play the game of Section \ref{algebras} in this setting. Given an extension of a fusion category $\Cc$ by a group $Q$ in the sense of \cite{eno2}, i.e., $Q\rightarrow\Bmm(\Cc)$ we have an action of $Q$ on $\Cc\text{-Mod}$ the $2$-category of module categories over $\Cc$, then one can write $$\Dd\text{-Mod}\simeq\Cc\text{-Mod}\,^Q.$$ This is similar to the point of view taken in \cite{clifford}.
\end{remark}

\begin{remark}
Observation \eqref{brpic} (or \cite{weaklygroup}) points out that $Q$-extensions of $\Cc$ in the sense of \cite{eno2}, when considered together, only depend on the Morita equivalence class of $\Cc$.  Thus they sit on the other extreme of the spectrum of notions of extensions, with maps to $\Att(\Cc)$ (or $\Ott(\Cc)$) having no duality invariance and maps to $\Bmm(\Cc)$ being invariant with respect to all dualities.  With the notion we suggest being in the middle, as possessing an invariance with respect to a chosen fixed duality.
\end{remark}

\begin{remark}
Since the Morita equivalence class of $\Cc$ is determined by its center $\Zz(\Cc)$, considered as braided fusion category \cite{weaklygroup}, it should be possible to restate the classification of $Q$-extensions of $\Cc$ in terms of suitable extensions of its center. This would replace Morita equivalences by actual braided equivalences.  More precisely, in \cite{eno2} it is shown that maps from $Q$ to $\Pc(\Bb)$ classify $Q$-braided extensions  of a braided fusion category $\Bb$. They are called braided $Q$-crossed categories with a faithful $Q$-grading with the trivial component $\Bb$ in \cite{eno2}. It is likely that there is an equivalence between $\Bmm(\Cc)$ and $\Pc(\Zz(\Cc))$ given by $\Mm\mapsto \Zz_{\Cc}(\Mm)=Fun_{\Cc\boxtimes\Cc^{rev}}(\Cc,\Mm)$, thus a correspondence between $Q$-extensions of $\Cc$ and $Q$-braided extensions of $\Zz(\Cc)$.  Under the bijection $\Dd$ corresponds to $\Zz_{\Cc}(\Dd)$ and $\Zz(\Dd)\simeq\Zz_{\Cc}(\Dd)^Q$, see \cite{centers}.  We mention this because (as shown in \cite{eno2}) $\underline{\text{Pic}}(\Zz(\Cc))\simeq\underline{\text{Aut}}^{Br}(\Zz(\Cc))$, so that $Q$-extensions of $\Cc$ are ``almost"\footnote{To be precise, given an $\omega\in H^3(Q)$ and a $Q$-extension $\Dd$ of $\Cc$, we can twist the associator of $\Dd$ by $\omega$ to form $\Dd^\omega$.  The collection of $\Dd^\omega$'s is exactly the $H^3(Q)$-torsor sitting over the morphism $Q\rightarrow \underline{\text{Aut}}^{Br}(\Zz(\Cc))$ that one obtains from $\Dd$.  Note that $\Zz(\Cc)\rtimes Q$ and $\Zz_{\Cc}(\Dd)$ need have little in common even after discounting the $\omega$-freedom.} maps from $Q$ to autoequivalences (monoidal and braided) of ``something".

\end{remark}

\subsection{The case of $\Cc=\ok$}
In the previous section we have defined a notion of a $Q$-extension of a pair $(\Cc,\Mm)$.  The motivation was to have a bijection between $Q$-extensions of a Hopf algebra $H$ and its dual Hopf algebra $H^*$.  This was achieved, but the extensions themselves turned out not to be Hopf algebras, nor even (canonically) weak Hopf algebras.  The purpose of this section is to demonstrate that despite this, at least in the case of $H=\Oo_K$ (and so $H^*=\C K$) it is a useful notion as the extensions it defines are very close to what one would want.

So let us study the bijection of $Q$-extensions for the case of $(\Cc=\ok,\vect)$ and $(\Cc^*=\rep(K),\vect)$.  In the Appendix of \cite{eno2} it is shown, among other things, that quasi-trivial $Q$ extensions of $\ok$, i.e., morphisms $Q\rightarrow \Ott(\ok)$ are given by the data of a group extension $G$ of $Q$ by $K$ plus a $3$-cocycle $\omega\in Z^3(G,\C^\times)$ that restricts to the trivial one on $K$.  More precisely, we can reinterpret the above as follows.

Let $\RP$ denote the $3$-group of fiber functors of $\ok$; it is a group because $\ok$ is a Hopf category.  Observe that $\RP$ is an abelian $3$-group given by the truncated group cochain complex $$C^0(K)\rightarrow C^1(K)\rightarrow Z^2(K)$$ where we suppress the trivial coefficients in the notation.  The $2$-group $B(K)$ acts on it naturally, i.e., an $X\in B(K)$ yields an invertible quasi-trivial $\ok$ bimodule category $\vect_X$ and if $\Mm$ is a $1$-dimensional module category over $\ok$, then so is $\vect_X\boxtimes_{\ok}\Mm$.  Then $$\Ott(\ok)\simeq \RP\rtimes B(K)$$ and $Q\rightarrow \RP\rtimes B(K)$ decomposes into a $Q\rightarrow B(K)$, i.e., an extension $G$ of $Q$ by $K$ and an element of $\drr(Q,\RP)$ the abelian $3$-group of derivations \cite{der} from $Q$ to its (via $Q\rightarrow B(K)$) module  $\RP$.  This $3$-group is given by the truncated relative group cochain complex $$C^1(G;K)\rightarrow C^2(G;K)\rightarrow Z^3(G;K).$$

So up to equivalence, the quasi-trivial $Q$-extensions of $\ok$  are given by pairs $(G,\omega)$ with $G$ an extension of $Q$ by $K$ and $\omega\in H^3(G;K)$ up to isomorphism of extensions $\phi:G\simeq G'$ with $\phi^*\omega'=\omega$.

\begin{definition}
Let us say that a quasi-trivial $Q$-extension of $\ok$ given by a pair $(G,\omega)$ is \emph{based} on $G$.
\end{definition}

Explicitly, given $(G,\omega)$ the corresponding quasi-trivial $Q$-extension is $\vect^\omega_G$ where by abuse of notation we denote by $\omega$ the lift of $\omega$ to an element of $Z^3(G;K)$.  This fusion category has an obvious $Q$-grading with the identity component canonically isomorphic to $\ok$.  Another lift of $\omega$ to a cocycle would yield an isomorphic quasi-trivial $Q$-extension.  The dual $Q$-extension of $\rep(K)$ is $S^\omega_{(K,0)}(G)$ (see Section \ref{stuff2} for the notation) with its canonical inclusion of $\rep(K)$ as the identity component.  Note that up to equivalence, as $\omega$ varies over $H^3(G;K)$, these include all of the fusion categories of the type $S^\phi_{(K,\eta)}(G)\simeq S_{(K,0)}^{\phi-d\widetilde{\eta}}(G)$ (also parameterized by $H^3(G;K)$).  As expected not all of these are quasi-trivial $Q$-extensions of $\rep(K)$ as we did not restrict ourselves to the homomorphisms $Q\rightarrow \Ott(\ok,\vect)$.

For example, the boundary map $H^2(K)\rightarrow H^3(G;K)$ describes the inclusion of $S_{\widehat{K}}(G)$'s into $S^\omega_{(K,0)}(G)$'s and the $Q$-graded components of the former can be readily seen to equal $\rep^{\eta-q\cdot\eta}(K)$ where $\widehat{K}=\widehat{K}^\eta$ with $\eta\in H^2(K)$.  Thus the only quasi-trivial $Q$-extensions of $\rep(K)$ among these are the one indexed by $H^2(K)^Q/im H^2(G)$.

For a general $S^\omega_{(K,0)}(G)$, in order to understand the structure of the $q$th graded component as a left $\rep(K)$ module, we need to examine the map \begin{align*}\alpha: H^3(G;K)&\rightarrow \text{Der}(Q, H^2(K))\\\omega&\mapsto\varphi_\omega.  \end{align*}  It is obtained from the consideration of the adjoint action of $G$ on $\vect_G^\omega$, i.e., $G\rightarrow\Att(\vect_G^\omega)$ which restricts to $G\rightarrow\Att(\ok)$.  Recall that $\pi_1(\Att(\ok))\simeq H^2(K)\rtimes\text{Aut}(K)$ and the resulting map $G\rightarrow H^2(K)\rtimes\text{Aut}(K)$ descends to $Q\rightarrow H^2(K)\rtimes\text{Out}(K)$ yielding a derivation $q\mapsto\varphi_\omega(q)$.  It is not hard to see that as left $\rep(K)$ modules $$S^\omega_{(K,0)}(G)\simeq\bigoplus_{q\in Q}\rep^{\varphi_\omega(q)}(K).$$

Thus the kernel of $\alpha$ parameterizes $Q$-extensions (in our sense) of $(\rep(K),\vect)$ and so of $(\ok,\vect)$ based on a fixed $G$.  Alternatively, $\alpha$ can be understood as follows. Observe that as $Q$-modules, we have a short exact sequence of abelian $3$-groups \begin{equation}\label{rep}\Rp[1]\rightarrow\RP\rightarrow H^2(K)\end{equation} where $\Rp[1]$ is an upshifted $2$-group of $1$-dimensional representations of $K$.  This induces a long exact sequence on $\pi_i(\drr(Q,-))$'s from which we get $$\xymatrix{0\ar[r]& H^2(Q,\Rp)\ar[r] &H^3(G;K)\ar[r]^{\!\!\!\!\!\!\!\!\!\!\alpha}&\text{Der}(Q, H^2(K)).}$$  Another observation that follows from the above considerations is $$\Ott(\ok,\vect)\simeq\Rp[1]\rtimes B(K).$$

Let us examine $H^2(Q,\Rp)$ more closely.  Roughly speaking, it classifies $Q$-extensions of $(\vect_{K^*},\vect)$ with a prescribed action of $Q$. More precisely, $Q$ acts on $\Rp=\vect_{K^*}^\times$ via $G$, i.e., $q\cdot\ell=G_q\times_K\ell$.   A manageable description of this data is extracted from applying the group cohomology (of $Q$) functor to the short exact sequence of $Q$-modules
$$\C^\times[1]\rightarrow\Rp\rightarrow K^*$$ obtaining \begin{equation}\label{map}H^1(Q,K^*)\rightarrow H^3(Q)\rightarrow H^2(Q,\Rp)\rightarrow H^2(Q,K^*)\rightarrow H^4(Q).\end{equation}  On the other hand, by applying $Q$-cohomology to \eqref{rep} we obtain: $$H^2(G)\rightarrow H^2(K)^Q\rightarrow H^2(Q,\Rp)\rightarrow H^3_K(G)\rightarrow H^1(Q, H^2(K))$$ where $H^3_K(G)$ is the kernel of $H^3(G)\rightarrow H^3(K)$. In fact we have the following exact commutative diagram illustrating the relationship between $S_{\widehat{K}}(G)$'s, $S^\omega_{(K,0)}(G)$'s and those among them that are quasi-trivial $Q$-extensions of $\rep(K)$:
$$
\xymatrix{
&0\ar[d]&0\ar[d]&0\ar[d]&\\
0\ar[r]&\dfrac{H^2(K)^Q}{im H^2(G)}\ar[r]\ar[d]& H^2(Q,\Rp)\ar[r]\ar[d]&\text{ker}(\beta)\ar[r]\ar[d]&0\\
0\ar[r]&\dfrac{H^2(K)}{im H^2(G)}\ar[r]\ar[d]& H^3(G;K)\ar[r]\ar[d]^\alpha&H^3_K(G)\ar[r]\ar[d]^\beta&0\\
0\ar[r]&\dfrac{H^2(K)}{H^2(K)^Q}\ar[r]\ar[d]& \text{Der}(Q, H^2(K))\ar[r]&H^1(Q, H^2(K))\ar[r]&0\\
&0&&&\\
}
$$

Yet another point of view is the most direct interpretation of the notation itself, namely consider $H^2(Q,\Rp)$ as the group of equivalence classes of $2$-cocycles, i.e., collections of $\ell_{q,q'}\in\Rp$ for all pairs $q,q'\in Q$ together with isomorphisms in $\Rp$: $$\varphi_{q,q',q''}:(G_q\times_K\ell_{q',q''})\otimes\ell_{qq',q''}^*\otimes\ell_{q,q'q''}\otimes\ell^*_{q,q'}\simeq\C$$ satisfying a twisted pentagon axiom.  We can then form an explicit $Q$-extension of $(\rep(K),\vect)$ based on $G$ consisting of pairs $(V,q)$ with  $V\in\rep(K)$ and $q\in Q$ and the fusion product \begin{equation}\label{fusion}(V,q)\star(W,q')=(V\otimes(G_q\times_K W)\otimes\ell_{q,q'},qq').\end{equation}

\begin{remark}
Note that the $Q$-extension described by \eqref{fusion} is dual to the  $Q$-extension of $(\ok,\vect)$ based on $G$ given by $\vect^\omega_G$ with $\omega\in H^3(G;K)$ satisfying $\alpha(\omega)=0$, under the identification of $\omega$ with $\ell_{q,q'}$'s.  It is simply $S^\omega_{(K,0)}(G)$.  Observe that by Proposition \ref{brpic} and \eqref{fusion} all $Q$-extensions of $(\rep(K),\vect)$ arise from the  $Q$-extensions of $(\rep(K_{ab}),\vect)$ equipped with a lifting of $Q\rightarrow B(K_{ab})$ to $Q\rightarrow B(K)$.
\end{remark}

For a fixed action $\rho$ of $Q$ on an abelian group $A$, a $Q$-extension of $(\vect_A,\vect)$ with band $\rho$ (see Remark \ref{band}) is given by an $\eta\in H^2(Q,A)$ specifying an extension $G_\eta$ (the base) and a $\omega\in H^3(G_\eta;A)$ with $\alpha(\omega)=0$, i.e., a pointed fusion category $\vect^\omega_{G_\eta}$.  Its dual is a $Q$-extension of $(\vect_{A^*},\vect)$ with band $\rho^*$.  More precisely, it is $S^\omega_{A,0}(G_\eta)$ just as before, but the latter is now also a pointed fusion category $\vect^{\omega^*}_{G_{\eta^*}}$.  Indeed, by the formula \eqref{fusion}, the fusion product is $$(\chi,q)\star(\chi',q')=(\chi\otimes(G_\eta)_q\times_A\chi'\otimes\ell^\omega_{q,q'},qq')$$ so that $\eta^*$ is the image in $H^2(Q,A^*)$ of $\omega\in H^2(Q,\rep^\times(A))$ (see \eqref{map}).  To summarize: $$(\vect^\omega_{G_\eta})^{*\vect_Q}\simeq\vect^{\omega^*}_{G_{\eta^*}}$$ which extends the formula \eqref{dual} to a complete duality.

\section{A curious symmetry}\label{symmetry}
Here we examine $Q$ actions by tensor auto-equivalences on $\ok$.  These can be understood as $\C^\times$-central extensions of a double group constructed out of $Q$, $K$, and an action of the former on the latter.  This point of view reveals an interesting symmetry between $Q$ actions on $\ok$ and $K$ actions on $\vect_Q$.

Recall that the $2$-group $\Att(\rep(K))$ is somewhat mysterious as in addition to the usual $K$-bitorsors, it contains their quantum analogues: $K$-bi-Galois algebras (see \cite{twistedauthopf} for example).  Its dual problem on the other hand, namely the consideration of $\Att(\ok)$ is simpler.  More precisely, we have: $$\Att(\ok)\simeq \text{Ext}(K)\rtimes Aut(K),$$ where $\text{Ext}(K)$ denotes the abelian $2$-group of $\C^\times$-central extensions of $K$. More concretely, it is given by the truncated group cohomology complex $$C^1(K)\rightarrow Z^2(K),$$ so that $$\text{Aut}^\otimes(\ok)\simeq H^2(K)\rtimes Aut(K).$$

Note that the homomorphisms $Q\rightarrow \Att(\ok)$ have a curious interpretation, namely for a fixed $\varphi:Q \rightarrow \text{Aut}(K)$ defining an action of $Q$ on $K$, homomorphisms that project to $\varphi$ are parameterized (up to equivalence) by $H^2(D(Q,K,\varphi))$ where $D(Q,K,\varphi)$ is the double group (see for example \cite{doublegroupoids} for the definitions) associated to $\varphi$.

In particular if we focus on $Q\rightarrow \Att_{\,0}(\ok)$, i.e., maps into the subgroup of auto-equivalences that preserve the isomorphism classes of objects in $\ok$, so that $\varphi$ is trivial, then we see that $$Hom(Q,\Att_{\,0}(\ok))\simeq Hom(K,\Att_{\,0}(\vect_Q))$$ so that actions of $Q$ on $\ok$ that fix isomorphism classes are in bijection with actions of $K$ on $\vect_Q$   that fix isomorphism classes.  More explicitly, either action consists of the same data (from different perspectives) of functions $$\omega:Q\times K\times K\rightarrow\C^\times,\quad\eta:K\times Q\times Q\rightarrow\C^\times$$ subject to the conditions $$\partial_q \omega =\partial_k\eta,\quad \partial_k\omega=0, \quad\partial_q\eta=0$$ where $\partial_q$ and $\partial_k$ are group cohomology differentials.

Observe that if both $H^2(Q)$ and $H^2(K)$ vanish, then the correspondence asserts that $$H^2(Q,K^*)\simeq H^2(K,Q^*).$$  This also follows from the universal coefficients theorem whereby the vanishing of $H^2$ establishes $H^2(Q,K^*)\simeq Ext^1(Q_{ab},K^*)$ and $H^2(K,Q^*)\simeq Ext^1(K_{ab},Q^*)$, and the ext-groups are isomorphic by the usual Pontryagin duality. Thus the correspondence can be seen as its ``generalization" to the case when $H^2$ does not vanish.

\begin{remark}
Let $\rho\in Hom(Q,\Att_{\,0}(\ok))$ and $r\in Hom(K,\Att_{\,0}(\vect_Q))$ correspond under the above, and note that the band of the latter is $\pi_1(r):K\rightarrow H^2(Q)$, then we have $$(\ok)^Q\simeq\bigoplus_{k\in K}\rep^{\pi_1(r)(k)}(Q)$$ as $\rep(Q)$-modules.
\end{remark}

\section{Appendix}\label{appx}

\subsection{Clifford-Mackey theory}\label{stuff0}
Let $G$ be a finite group with a normal subgroup $K$.  Denote the quotient by $Q$.  Observe that for $q\in Q$, the fiber $G_q$ in $G$ over $q$ is a $K$-bitorsor, i.e., it has two commuting $K$-actions, one left, one right, such that both are simply transitive. Bitorsors form a $2$-group that can be presented as a crossed product $B(K)=[K\rightarrow\text{Aut}(K)]$.  Their usefulness to us stems from the tensor compatible action   of $B(K)$ on $\rep(K)$ as follows: for an $X\in B(K)$ and $V\in \rep(K)$ we have $$X\cdot V = X\times_K V$$ where the latter is by definition $$\dfrac{X\times V}{(xk,v)\sim(x,kv)}.$$  Thus we have a tensor compatible action of $Q$ on $\rep(K)$ obtained from $K\lhd G$. Let us make the situation even more manageable: if $I_K$ denotes the set of isomorphism classes of irreducible representations of $K$ on complex vector spaces, then $Q$ acts on $I_K$.  Furthermore, if we choose a representative $V_i$ for every $i\in I_K$, and let $$\ell_{q,i}=\Hom_K(V_{q\cdot i},G_q\times_K V_i)$$ then the action groupoid $Q\act I_K$ has a ``decoration" of lines, with $\ell_{q,i}$ assigned to the arrow $(q,i)$ from $i$ to $q\cdot i$.  We have the usual associative composition isomorphisms that give the data of an $S^1$-gerbe on  $Q\act I_K$.  This is a special case of the situation considered in \cite{gerbes}.  It is a reinterpretation in the language of gerbes of the Clifford-Mackey theory.

\subsection{Some fusion categories and their modules}\label{stuff1}
Let $G$ be a finite group with an arbitrary subgroup $K$. We have a fusion category $S_K(G)$ that consists of $K$-bi-equivariant finite dimensional $G$-graded vector spaces over $\C$.  More precisely, an element $M\in S_K(G)$ is graded $M=\bigoplus_{g\in G}M_g$ and equipped with isomorphisms $k\cdot:M_g\rightarrow M_{kg}$ and $\cdot k:M_g\rightarrow M_{gk}$ subject to the obvious compatibility conditions.  The convolution product is defined by $$(M\star N)_g=\left(\bigoplus_{g_1 g_2=g}M_{g_1}\otimes N_{g_2}\right)^K$$ where the $k:M_{g_1}\otimes N_{g_2}\rightarrow M_{g_1 k^{-1}}\otimes N_{kg_2}$. Alternatively, $S_K(G)$ is the category of $G$-equivariant (with respect to the diagonal action) finite dimensional $G/K\times G/K$-graded vector spaces.  As special cases we obtain both $\vect_G$ (the $G$-graded vector spaces) and $\rep(G)$ (the finite dimensional representations of $G$) from $K=1$ and $K=G$ respectively.

There is an alternative description \cite{drinfelddouble} of these objects in terms of modules and bimodules in fusion categories.  More precisely, $\vect_{G/K}$ is a $\vect_G$-module and as such $\vect_{G/K}\simeq\Mod_{\vect_G}\C K$.  Furthermore, one checks that $S_K(G)\simeq\text{Bimod}_{\vect_G}\C K$.

If we suppose that $K$ is a normal subgroup with the quotient group $Q$, then $$S_K(G)=\rep(K)\rtimes Q$$ more precisely, for $V,W\in\rep(K)$ and $q,q'\in Q$ we have $$(V,q)\otimes(W,q')=(V\otimes q\cdot W,qq').$$  The action of $Q$ on $\rep(K)$ is via the bitorsors $G_q$, as above.  In particular, $\vect_Q$ is a full fusion subcategory of $S_K(G)$ in this case.

Assume that $K=A$ is an abelian normal subgroup then let $\omega\in H^2(Q,A)$ be the cohomology class classifying the extension $$A\rightarrow G\rightarrow Q.$$  Namely, if we choose a set theoretic section $s$ of the extension and define $$f(q,q')=s(q)s(q')s(qq')^{-1}$$ then $f$ is a $2$-cocycle representing $\omega$.  Moreover, $Q$ acts on $A^*$ (the characters of $A$) on the right and so we can form a semi-direct product $Q\ltimes A^*$.  We have a map \begin{align*}C^2(Q,A)&\rightarrow C^3(Q\ltimes A^*,\C^{\times})\\f&\mapsto\phi_f \end{align*} where $$\phi_f((q,\chi),(q',\chi'),(q'',\chi''))=\chi(f(q',q'')).$$  One readily checks that it sends cocycles to cocycles and we get $$S_A(G)\simeq\vect^{\phi_f}_{Q\ltimes A^*}$$ where the latter is the fusion category of $Q\ltimes A^*$-graded vector spaces with the associator twisted by the $3$-cocycle $\phi_f$.

\subsection{A very twisted case}\label{stuff2}
Let $\omega\in Z^3(G,\C^\times)$; for construction purposes we need a $3$-cocycle, however up to equivalence only the cohomology class of $\omega$ will matter.  As in \cite{eno1} for example, we can construct a fusion category $\vect_G^\omega$ which is almost  $\vect_G$, but with the associator twisted by $\omega$.  More precisely we still have that $\delta_g\otimes\delta_{g'}=\delta_{gg'}$, but the associator $\delta_g\otimes(\delta_{g'}\otimes\delta_{g''})\rightarrow (\delta_g\otimes\delta_{g'})\otimes\delta_{g''}$ is not $Id_{\delta_{gg'g''}}$, but rather $\omega(g,g',g'')Id_{\delta_{gg'g''}}$.

Suppose that $K$ is a subgroup of $G$ equipped with a cochain $\eta\in C^2(K,\C^\times)$ such that $$d\eta=\omega|_K.$$  One may construct an embedding $$\eta:\vect_K\rightarrow \vect_G^\omega$$ by using $\eta$ to modify the fusion compatibility isomorphism. Namely, $\eta(\delta_k)=\delta_k$, but $\eta(\delta_k)\otimes \eta(\delta_{k'})\rightarrow\eta(\delta_{kk'})$ is not $Id_{\delta_{kk'}}$, but $\eta(k,k')Id_{\delta_{kk'}}$.

Though it is possible to proceed more geometrically as we have done above, it becomes increasingly more unwieldy.  Thus we proceed algebraically.  Note that if we modify the multiplication in the group algebra $\C K$, by $\eta$, i.e., we set $k\cdot k'=\eta(k,k')kk'$ then the resulting algebra $\C_\eta K$ is not associative in $\vect$ nor $\vect_G$.  However it is easy to check that $\C_\eta K$ is an associative algebra in $\vect_G^\omega$.

\begin{definition}
Let $$\vect_{G/K}^\eta=\Mod_{\vect_G^\omega}\C_\eta K,$$   $$S^\omega_{(K,\eta)}(G)=\Bimod_{\vect_G^\omega}(\C_\eta K),$$ $$S^\omega_{(K,\eta),(T,\nu)}(G)=\Bimod_{\vect_G^\omega}(\C_\eta K,\C_\nu T),$$ where $T$ is a subgroup of $G$ together with a cochain $\nu\in C^2(T,\C^\times)$ such that $d\nu=\omega|_T$.
\end{definition}

\bigskip
\noindent{\bf Acknowledgments.} We would like to thank Dmitri Nikshych for patiently answering our questions.

\bibliography{fusion}{}

\begin{thebibliography}{10}

\bibitem{doublegroupoids}
Nicol{\'a}s Andruskiewitsch and Sonia Natale.
\newblock Tensor categories attached to double groupoids.
\newblock {\em Adv. Math.}, 200(2):539--583, 2006.

\bibitem{fusionrules}
Sebastian Burciu and Sonia Natale.
\newblock Fusion rules of equivariantizations of fusion categories.
\newblock {\em J. Math. Phys.}, 54(1):013511, 21, 2013.

\bibitem{obstructions}
Antonio~M. Cegarra and Antonio~R. Garz{\'o}n.
\newblock Obstructions to {C}lifford system extensions of algebras.
\newblock {\em Proc. Indian Acad. Sci. Math. Sci.}, 111(2):151--161, 2001.

\bibitem{twistedauthopf}
A.~Davydov.
\newblock Twisted automorphisms of {H}opf algebras.
\newblock In {\em Noncommutative structures in mathematics and physics}, pages
  103--130. K. Vlaam. Acad. Belgie Wet. Kunsten (KVAB), Brussels, 2010.

\bibitem{eno2}
Pavel Etingof, Dmitri Nikshych, and Victor Ostrik.
\newblock Fusion categories and homotopy theory.
\newblock {\em Quantum Topol.}, 1(3):209--273, 2010.
\newblock With an appendix by Ehud Meir.

\bibitem{weaklygroup}
Pavel Etingof, Dmitri Nikshych, and Victor Ostrik.
\newblock Weakly group-theoretical and solvable fusion categories.
\newblock {\em Adv. Math.}, 226(1):176--205, 2011.

\bibitem{eno1}
Pavel Etingof, Dmitri Nikshych, and Viktor Ostrik.
\newblock On fusion categories.
\newblock {\em Ann. of Math. (2)}, 162(2):581--642, 2005.

\bibitem{clifford}
C{\'e}sar Galindo.
\newblock Clifford theory for tensor categories.
\newblock {\em J. Lond. Math. Soc. (2)}, 83(1):57--78, 2011.

\bibitem{der}
A.~R. Garz{\'o}n, H.~Inassaridze, and A.~del R{\'{\i}}o.
\newblock Derivations of categorical groups.
\newblock {\em Theory Appl. Categ.}, 13:No. 5, 86--105, 2004.

\bibitem{centers}
Shlomo Gelaki, Deepak Naidu, and Dmitri Nikshych.
\newblock Centers of graded fusion categories.
\newblock {\em Algebra Number Theory}, 3(8):959--990, 2009.

\bibitem{ni}
Dmitri Nikshych.
\newblock Non-group-theoretical semisimple {H}opf algebras from group actions
  on fusion categories.
\newblock {\em Selecta Math. (N.S.)}, 14(1):145--161, 2008.

\bibitem{modulecategories}
Victor Ostrik.
\newblock Module categories, weak {H}opf algebras and modular invariants.
\newblock {\em Transform. Groups}, 8(2):177--206, 2003.

\bibitem{drinfelddouble}
Viktor Ostrik.
\newblock Module categories over the {D}rinfeld double of a finite group.
\newblock {\em Int. Math. Res. Not.}, (27):1507--1520, 2003.

\bibitem{gerbes}
Ilya Shapiro, Xiang Tang, and Hsian-Hua Tseng.
\newblock On the relative dual of an {$S^1$}-gerbe over an orbifold.
\newblock {\em Adv. Math.}, 270:1--20, 2015.

\end{thebibliography}
\bibliographystyle{plain}

\medskip
\noindent Department of Mathematics and Statistics,
University of Windsor, 401 Sunset Avenue, Windsor, Ontario N9B 3P4, Canada

\noindent\emph{E-mail address}:
\textbf{ishapiro@uwindsor.ca}

\end{document}